\documentclass[12pt, a4paper]{article}

\usepackage{amsmath, amsthm, amsfonts, mathtools, bbm, graphicx, tikz}
\usepackage[font=small]{caption}
\usepackage[round]{natbib}
\usepackage{subcaption}
\usepackage{float}
\usetikzlibrary{matrix}

\newcommand{\reals}{\mathbb{R}}

\newtheorem{mythm}{Theorem}[section]

\theoremstyle{definition}

\numberwithin{equation}{section}
\theoremstyle{plain}

\begin{document}

\title{Estimation of the $r$-th derivative of a density function by the tilted kernel estimator}
\date{}
\maketitle
\begin{abstract}
We consider the problem of estimating the $ s $-th derivative of a density function $ f $ by the \textit{tilted} Kernel estimator introduced in \cite{Hall_Doosti}. Then we further show this estimator achieves the same convergence rate, in probability, the wavelet estimators achieved as shown in \cite{Hall_Patil_1995}. That is, the convergence rate of $ O_p \left( n^{-(r-s)/(2r+1)} \right) $.
\end{abstract}

\section{Tilted Kernel estimators for derivative}

\subsection{Settings and assumptions}

Suppose we have an $ i.i.d $ sample of random variables $ X_1, X_2, \ldots, X_n $. \cite{Hall_Doosti} introduced the \textit{tilted} Kernel estimator $\hat{f}_n(x) $ as

\begin{equation} \label{Tilted Kernel Estimator}
\hat{f}_n(x) = \frac{1}{h} \sum_{i=1}^{n} p_{i,n} K \left( \frac{x - X_i}{h} \right),
\end{equation}
where $ \sum_{i=1}^n p_{i,n} = 1 $, $ K $ is bounded, symmetric and compactly supported. Note that the conventional Kernel estimator is a special case of when $ p_{i,n} = 1/n $ for all $ i = 1, 2, \ldots, n $.

Observe that $ p_n = \{ p_{1,n}, p_{2,n}, \ldots, p_{n,n} \} $ forms a probability measure on $ n $ points for $ n \in \mathbb{N} $, let us agree to call the family $ \{p_n : n \in \mathbb{N} \} $ the \textit{family of probability measures associated to} $ \hat{f}_n $.

Suppose we know the the density function $ f $ is $ r $ times differentiable, it would be intuitive to just take the $ s $-th derivative of $ \hat{f}_n $ to be the estimator of $ f^{(s)} $, for $ s \leq r $, namely,

\begin{equation}
\hat{f}^{(s)}_n (x) = \frac{1}{h} \sum_{i=1}^{n} p_{i,n} \frac{\partial^s}{\partial x^s} K \left( \frac{x - X_i}{h} \right).
\end{equation}

\subsection{Selection of the associated probability measures}

In \cite{Hall_Doosti}, the authors selected the associated probability measures to be, for $ n \in \mathbb{N} $,

\begin{equation} \label{associated measures}
p_{i,n} = \frac{1 + h_n^2 g(X_i)}{n (1 + \Delta)}, \quad \text{where} \quad \Delta_n = \frac{h_n^2}{n} \sum_{i=1}^{n} g(X_i), \quad i = 1, ,2 \ldots, n,
\end{equation}
with $ h_n $ the bandwidth and $ g $ defined as follow,

\begin{equation}
g = \frac{1}{2} \frac{f''}{f} + h^2 \frac{e_2}{f} + \ldots + h^{2r-2} \frac{e_{2r-2}}{f}.
\end{equation}
The functions $ e_2, e_4, \ldots, e_{2r-2} $ are obtained recursively as in \cite{Hall_Doosti}. Consequently, we can express the expectation of $ \hat{f}_n $ as,
\begin{align*}
E\left( \hat{f}_n(x) \right) = \psi(x) = f(x) + \frac{h^{r-s}}{(r-1)!} f^{(r)} (\xi) \int u^r K(u) du.
\end{align*}

\subsection{Convergence speed}
Suppose we are estimation the $ s $-th derivative of a density. It is shown in \cite{Hall_Patil_1995} that the convergence rate, in probability, of $ n^{-(r-s)/(2s+1)} $ can be achieved by the wavelet estimator.

We show that, in Theorem \ref{speed}, the same convergence speed can be achieved by the tilted kernel estimator with an appropriate associated probability measures. That is, we show that
\begin{equation}
\hat{f}^{(s)}_n(x) - f^{(s)}(x) = O_p \left( n^{-\frac{r-s}{2r+1}} \right),
\end{equation}
with the appropriate choice of the bandwidth $ h $.

\section{Technical results}

\begin{mythm} \label{speed}
Suppose $ f $ is a $ r $-th time differentiable density function where $ r $ is even . Assume $ f^{(2)}, f^{(4)}, \ldots, f^{(r)} $ and $ f^{(2)}/f, f^{(4)}/f, \ldots, f^{(r)}/f $ exist and are bounded on $ \reals $. Then there exists $ h_n > 0 , n \in \mathbb{N} $ and a set of  probability measures $ \{ \mu_n : n \in \mathbb{N} \} $ associated to $ \hat{f}_n $ such that for all $ x $ and $ s \leq r $,
\begin{equation}
\hat{f}^{(s)}_n(x) - f^{(s)}(x) = O_p \left( n^{-\frac{r-s}{2r+1}} \right)
\end{equation}
where $  $
\end{mythm}

\begin{proof}
Let $ \tilde{p}_{i,n} = n^{-1}\left(1 + h^{2}g(X_i) \right) $, the non-standardised version of $ p_{i,n} $ as in (\ref{associated measures}) and 
\begin{equation}
\tilde{f}_n(x) = \frac{1}{h} \sum_{i=1}^{n} \tilde{p}_{i,n} K \left(\frac{x - X_i}{h} \right).
\end{equation}

Since $$  \left( \hat{f}_n(x) - f(x) \right)^2 \leq   2 \left[ \left( \hat{f}_n(x) - \tilde{f}_n(x) \right)^2 +  \left( \tilde{f}_n(x) - f(x) \right)^2\right], $$ the result can be shown by showing
\begin{align}
 \left( \hat{f}^{(s)}_n(x) - \tilde{f}^{(s)}_n(x) \right)^2 = O_p \left( n^{-\frac{2(r-s)}{2r+1}} \label{1} \right)\\
 \left( \tilde{f}^{(s)}_n(x) - f^{(s)}(x) \right)^2 = O_p \left( n^{-\frac{2(r-s)}{2r+1}} \label{2} \right)
\end{align}
for some $ h = h_n >0 $ and probability measures $ \{ \mu_n : n \in \mathbb{N}\} $.

To show (\ref{1}), first let $  \Delta_n = n^{-1} \sum_{i=1}^{n} h^2 g(X_i) $ and observe that by choosing the associated measure as in (\ref{associated measures}), we have
\begin{align} \label{residual}
&\left( \hat{f}^{(s)}_n(x) - \tilde{f}^{(s)}_n(x) \right)^2 \nonumber \\ = 
&\left( 1 - \frac{1}{1 + \Delta_n} \right)^2 \left[ \frac{1}{nh} \sum_{i=1}^{n} \left( 1 + h^2 g(X_i) \right) \frac{\partial}{\partial x} K\left( \frac{x - X_i}{h} \right) \right]^2
\end{align}

Since $ g $ is, by definition and assumption, bounded on the real line, $ \Delta_n \leq 1/2 $ almost surely for sufficiently large $ n $. Therefore the first term on the right hand side of (\ref{residual}) can be bounded, almost surly, by $ 4 \Delta_n^2 $ and hence
\begin{align*}
E \left( 1 - \frac{1}{1 + \Delta_n} \right)^2 \leq 4 E \Delta_n^2 = O\left(n^{-1}\right),
\end{align*}
where the last equality follows from the Law of Large Numbers applied to the triangular array $ \{ h^2_i g(X_j) : i, j \in \mathbb{N} \} $. Thus,
\begin{align*}
\left( 1 - \frac{1}{1 + \Delta_n} \right)^2 = O_p \left( n^{-1} \right) = O_p \left( n^{-\frac{2(r-s)}{2r+1}} \label{2} \right)
\end{align*}
Again, since $ g $, $ K^{(r)} $ are bounded, the expectation of the second term of (\ref{residual}) can be bounded as follow,
\begin{align*}
& E  \left[ \frac{1}{nh} \sum_{i=1}^{n} \left( 1 + h^2 g(X_i) \right) \frac{\partial}{\partial x} K\left( \frac{x - X_i}{h} \right) \right]^2 \\
\leq & C_1 \left[ Var \tilde{f}^{(s)}_n(x) - \left( E \tilde{f}^{(s)}_n(x) \right)^2 \right]
\end{align*}
Thus we have shown (\ref{1}).

To show (\ref{2}), let $ \psi(x) = E \left( \tilde{f}^{(s)}_n(x) - f^{(s)}(x) \right) $, note that
\begin{align*}
& E \left( \tilde{f}^{(s)}_n(x) - f^{(s)}(x) \right)^2 \\ 
= & \psi(x)^2 + \frac{1}{nh^2} Var \left[ \left( 1 + h^2 g(X_i) \right) \frac{\partial}{\partial x} K\left( \frac{x - X_i}{h} \right) \right] \\
\leq & \psi(x)^2 + \frac{C_1^2}{nh^2} Var \left[ \frac{\partial}{\partial x} K\left( \frac{x - X_i}{h} \right) \right].
\end{align*}
Hence it suffices to show that the result hold for $ \psi $. To see this, we observe that as 
\begin{align*}
\frac{\partial^s}{\partial x^s} K\left( \frac{x - X_i}{h} \right) = h^{-s} K^{(s)} \left( \frac{x - X_i}{h} \right).
\end{align*}
Therefore
\begin{align*}
\psi(x) = f(x) - f^{(s)}(x) + \frac{h^{r-s}}{(r-1)!} f^{(r)} (\xi) \int u^r K(u) du,
\end{align*}
for some $ \xi \in [x,x+hu] $. Thus, by choosing $ h_n = n^{- \frac{1}{2r + 1}} $ we have
\begin{align*}
E \left( \tilde{f}^{(s)}_n(x) - f^{(s)}(x) \right)^2 = O \left( n^{-\frac{2(r-s)}{2r+1}}\right),
\end{align*}
hence proving (\ref{2}).
\end{proof}

\bibliographystyle{plainnat}
\bibliography{ref}

\end{document}